\documentclass[10pt]{amsart}
\usepackage{amsmath,amsthm,amssymb,amsfonts, eucal, amscd, mathbbol, mathrsfs}
\usepackage{cite}
\usepackage{setspace}
\usepackage[all,cmtip]{xy}
\usepackage{graphicx}
\usepackage{subfig}
\usepackage{fancyhdr}
\usepackage{latexsym}
\usepackage{epic}
\usepackage{ifthen}
\usepackage[bookmarks,bookmarksnumbered, plainpages=false, pdfpagelabels]{hyperref}
\usepackage{rotating}
\usepackage{scalefnt}

\newtheorem{thm}{Theorem}
\newtheorem{cor}[thm]{Corollary}
\newtheorem{lem}[thm]{Lemma}

\theoremstyle{definition}
\newtheorem{defn}[thm]{Definition}
\theoremstyle{definition}

\theoremstyle{definition}

\theoremstyle{definition}

\theoremstyle{definition}

\theoremstyle{definition}

\theoremstyle{definition}

\numberwithin{thm}{section}

\newcommand{\R}{\ensuremath{\mathbb{R}}}

\newcommand{\C}{\ensuremath{\mathbb{C}}}

\def\p{\partial}

\def\i{\infty}

\def\e{\epsilon}

\def\supp{\it{supp}}

\def\l{\lambda}

\def\o{\omega}

\def\B{\mathcal{B}}
\def\hB{\hat{\mathcal{B}}}
\def\P{\mathcal{P}}

\def\cal{\mathcal}

\def\rotateminus{\reflectbox{\rotatebox[origin=c]{155}{\hspace{.6pt}-}}}

\def\Xint#1{\mathchoice
{\XXint\displaystyle\textstyle{#1}}%
{\XXint\textstyle\scriptstyle{#1}}%
{\XXint\scriptstyle\scriptscriptstyle{#1}}%
{\XXint\scriptscriptstyle\scriptscriptstyle{#1}}%
\!\int}
\def\XXint#1#2#3{{\setbox0=\hbox{$#1{#2#3}{\int}$}
\vcenter{\hbox{$#2#3$}}\kern-.5\wd0}}

\def\cint{\Xint    \rotateminus }

\begin{document}
	\title[Cauchy Integral Theorems]{Generalizations of the Cauchy Integral Theorems}
\author[J. Harrison \& H. Pugh]{J. Harrison \\Department of Mathematics \\University of California, Berkeley\\ \\ H. Pugh \\  Department of Pure Mathematics and Mathematical Statistics\\Cambridge University}
\date{December 27, 2010}

\maketitle

\begin{abstract}
	We extend the Cauchy residue theorem to a large class of domains including differential chains that represent, via canonical embedding into a space of currents, divergence free vector fields and non-Lipschitz curves.  That is, while the classical Cauchy theorems involve integrals over piecewise smooth parameterized curves, these classical theorems actually hold for far more general notions of ``curve.'' We also extend the definition of winding number to these domains and show that it behaves as expected.  
\end{abstract}

\section{Introduction} 
\label{sec:introduction}



A systematic method is in place \cite{harrison1, continuity,  OCI, poincarelemma}  to treat pairs of  $k$-dimensional domains and differential $k$-forms in open subsets $U$ of  Riemannian $n$-manifolds $M$, for $0 \le k \le n$, chosen from locally convex spaces of ``differential chains'' $\hB(U)$ and differential forms $\B(U)$, with a jointly continuous integral $\cint_J \o$.  The integral is well-defined and canonical because of a de Rham isomorphism theorem at the level of cochains and forms 
 $\hB(U)' \cong \B(U)$  (Theorem 3.3.8 of \cite{OCI}), that passes to the classical de Rham isomorphism of cohomology classes.  Permitted boundaries include non-Lipschitz domains such as curves with countably many corners and cusps, stratified sets, and non-manifold smooth boundaries such as divergence free vector fields.

Differential chains $\hB(U)$ form a separable l.c.s. and its image under the canonical injection into its bidual is a proper subspace of currents $\B(U)'$ when $U=\R^n$ (and more generally, whenever $\B(U)$ itself is not reflexive \cite{topological}).  The advantage to using this subspace is that differential chains can be strongly approximated by polyhedral chains and pointed chains.

  In this paper we establish generalizations of the Cauchy Integral Theorem, the Cauchy Integral Formula, and the Cauchy Residue Theorem to such domains.  Central to our arguments are three chain operators pushforward $F_*$, boundary $\p$, and the cone operator $K$  used in the Poincar\'e Lemma for differential chains \cite{poincarelemma}.  These results first appeared as part of H.Pugh's senior thesis \cite{thesis}.  The authors would like to thank M.W. Hirsch for his feedback and suggestions.

\section{Integral of complex forms over differential chains} 
We recall basic terms defined in \cite{OCI}: Let $U \subset \R^n$ be open.  Then  $\B_k^r(U)$ is the subspace of real-valued differential forms defined on $U$ of class $C^r$ whose elements have a bound on each derivative of order $\le r$, and  $\B_k(U) = \varprojlim \B_k^r(U)$, is the Fr\'echet space of $C^\infty$ forms on $U$ whose derivatives are all bounded.   The space $\P_k(U)$ is that of ``pointed $k$-chains,'' sections of the $k$-th exterior power of the tangent bundle of $\R^n$ finitely supported in $U$.  The space $\hB_k^r(U)$ is the completion of $\P_k(U)$ with a norm described in \cite{OCI}, and $\hB_k(U) = \varinjlim \hB_k^r(U)$ is the space of \emph{differential $k$-chains} endowed with the direct limit topology. The spaces $\hB_k^r(U)$ and $\B_k^r(U)$ are defined independently and $(\hB_k^r(U))' \cong \B_k^r(U)$, yielding an integral $\cint_J \o := \o(J)$ for $J \in \hB_k^r(U)$ and $\o \in \B_k^r(U)$, or $J \in \hB_k(U), \o \in \B_k(U)$.  

The support of a differential chain $J$, denoted $\supp(J)$ is defined to be the smallest closed subset $K$ of $\R^n$ such that $\cint_J\o=0$ for any form $\o$ supported in $\R^n\setminus K$.  We will use extensively the property that any $J\in \hB_k^r(\R^n)$ can be considered as an element of $\hB_k^r(U)$ where $U$ is any neighborhood of $\supp(J)$.  In particular, if $J$ is supported in $U$, then $J\in \hB_k^r(W)$, where $W$ is a neighborhood of $\supp(J)$, and the closure of $W$ is contained in $U$.

Our goal in this paper is to generalize the Cauchy residue theorem by generalizing the domains of integration, which are classically piecewise smooth parameterized closed curves.  These curves are examples of differential chains in the space $\hat{\cal{B}}_1(\C)$ (where $\C$ is treated in this context as the real vector space $\R^2$.) By virtue of the topological vector space $\B$ of forms we will work with, we need to stay away from places where the function blows up, either at singularities in the plane, or at infinity.

In other words, we wish to integrate a complex differential 1-form $f(z)dz$ defined on some open subset $U$ of $\C$.  The only requirement for the domain $J\in \hB_1(U)$ is that if $f=u+iv$, where $u$ and $v$ are real-valued functions, then $u$ and $v$ restrict to elements of $\B_1(W)$ where $W$ is any neighborhood of $\supp(J)$.  This is satisfied, for example, if $f$ is holomorphic and $J$ is compactly supported in $U$.  Define\footnote{One may also extend the integral to complex chains $J + iK$ where $J, K \in \hB_k$ as follows $\cint_{J + iK} \o + i \eta = \cint_{J} \o - \cint_K \eta + i (\cint_K \o + \cint_J \eta)$, but a full treatment of complex chains (involving complexified tangent spaces and pointed chains) will appear in a sequel.}

	\[ \cint_J f(z)dz  :=  \cint_{J}\left( u(x,y) dx  -v(x,y)dy\right)+ i\cint_{J} \left(v(x,y) dx + u(x,y)dy\right). \]

	\section{Cauchy Integral Theorem for Differential Chains}

\begin{thm}{Cauchy Integral Theorem for Differential Chains}
\label{Generalized Cauchy Integral Theorem}

Let $U\subset \C$ be a bounded contractible open set, let $f: U\rightarrow \C$ be a holomorphic function, and let $J \in \hB_1(\C)$ be supported in $U$, with $\p J=0$.  Then
\[
\cint_J f(z)dz=0.
\]
\end{thm}

\begin{proof}
  Apply Theorem 7.0.16 of \cite{poincarelemma}
to $J$ to get $J=\partial K$ where $K\in \hB_2(\C)$ supported in $U$.  By Stokes' Theorem for Differential Chains (Theorem 5.2.4 of \cite{OCI}), we get

\begin{align*}
\cint_J f(z)dz&=\cint_{\partial K} \left(u(x,y) dx  -v(x,y)dy\right)+ i\cint_{\partial K}\left( v(x,y) dx + u(x,y)dy\right)\\
&= \cint_K  \left(\frac{\p u}{\p y} + \frac{\p v}{\p x}\right) dydx + i \cint_K \left(\frac{\p v}{\p y} - \frac{\p u}{\p x}\right)dydx\\
&=0,
\end{align*}
where the final equality is given by the Cauchy-Riemann Equations.
\end{proof}

Theorem \ref{Generalized Cauchy Integral Theorem} implies the classical Cauchy Integral Theorem, because of the natural representation of a smooth curve as a differential 1-chain (\S 3.1 of \cite{OCI}).  But we can also integrate over more exotic domains such as non-rectifiable curves, Lipschitz curves, and divergence-free vector fields, again treating these objects as differential $1$-chains.  See  \S 3 of \cite{OCI} for examples of such domains.

So that we may state a generalized Cauchy residue theorem, we now give a definition of winding number for differential chains.

\section{Winding Number for Differential Chains}
\begin{defn}[Winding Number for Differential Chains]
\label{Generalized Winding Number}	

Let $J\in \hB_1(\C)$, and let $z\in \supp(J)^c$.  Then the winding number of $J$ about $z$, $\textrm{Ind}_J(z)$ is defined to be
\[
\textrm{Ind}_J(z):=\frac{1}{2\pi i}\cint_J \frac{dw}{w-z}.
\]
\end{defn}

Note that $f(w)=\frac{1}{w-z}\in \B_0(U)$, where $U$ is any neighborhood of $\supp(J)$ whose closure does not contain $z$.   It follows that the above integral is well-defined.  Via the representations of classical domains (\S 3 \cite{OCI}), it follows that definition \ref{Generalized Winding Number} corresponds to the classical definition where it is defined.  That is, when the differential chain $J$ corresponds to a piecewise differentiable, parametrized, closed curve, the above integral is equal to its classical counterpart.  However, we need to check that our $\textrm{Ind}_J(z)$ behaves nicely when extended to differential chains in general.  Immediately we see on connected components of $\supp(J)^c$ that $\textrm{Ind}_J(z)$ is continuous. This follows since $\frac{1}{w-z_i}\rightarrow \frac{1}{w-z}$ in $\B_0(U)$ when $z_i\rightarrow z$, $z_i\in \supp(J)^c$, $z_i\notin \bar{U}$.  We will show further in Theorem \ref{Winding number constant on connected components} that if $\p J=0$, then $\textrm{Ind}_J(z)$ constant on connected components of $\supp(J)^c$, and in Corollary \ref{Winding number is zero on unbounded connected component} that if $J$ is closed and compactly supported, then $\textrm{Ind}_J(z)$ is zero on the unbounded connected component of $\supp(J)^c$. But first, we have an immediate result, a generalized version of the Cauchy Integral Formula:

\section{Cauchy Integral Formula for Differential Chains}
\begin{thm}{Cauchy Integral Formula for Differential Chains}
\label{Generalized Cauchy Integral Formula}

Let $U\subset \C$ be a bounded contractible open set, $f: U\rightarrow \C$ holomorphic, $J\in \hB_1(\C)$ supported in $U$, and $z\in U\setminus\supp(J)$.  Then

\[
\textrm{Ind}_J(z)f(z)=\frac{1}{2\pi i}\cint_J \frac{f(w)}{w-z}dw.
\]
\end{thm}

\begin{proof}
	
The function
\[w\rightarrow
\begin{cases}
	\frac{f(w)-f(z)}{w-z} & \textrm{for } w\in U\setminus \{z\}\\
	f'(z)& \textrm{for } w=z
\end{cases}
\]
is holomorphic in $U$, so by the Cauchy Integral Theorem for Differential Chains,
\[
\cint_J \frac{f(w)-f(z)}{w-z} dw =0.
\]
The theorem follows from our definition of $\textrm{Ind}_J(z)$.
\end{proof}

\section{Properties of $\textrm{Ind}_J(z)$}
To show that $\textrm{Ind}_J(z)$ is well behaved, it is useful, if $J$ is closed, to approximate $J$ with a sequence of closed polyhedral chains\footnote{A polyhedral chain is a formal sum of weighted $k$-polyhedra.}.  Since the subspace of polyhedral chains is dense in $\hB_k(\R^n)$ (see Theorem 3.2.4 of \cite{OCI}), we know that $J$ can be approximated by polyhedral chains.  However, to insist that these polyhedral chains be closed is a strong statement that we cannot make just yet.  In fact, we will need something slightly stronger: we need the closed polyhedral chain approximation to avoid the point around which we are computing the winding number.  What follows in the next two lemmas is a proof of the existence of such a closed polyhedral chain approximation.

\begin{lem}\label{topdim}\footnote{Thanks to M. Hirsch for his help with this lemma.}
	Let $K\in \hB_m(S^m)$, where $S^m$ is the $m$-sphere.  If $\p K=0$, then $K=a \hat{S}^m$ for some $a\in \R$, where $\hat{S}^m\in \hB_m(S^m)$ denotes the differential chain canonically associated to $S^m$.
\end{lem}
\begin{proof}
	The space $\hB_m(S^m)$ naturally embeds (see \cite{topological}) in the space of de Rham currents on $S^m$ via inclusion into the bidual, $\hB_m(S^m)\hookrightarrow (\B_m(S^m))'$.  Boundary commutes with this map, so we get a closed $m$-current $\hat{K}$ associated to $K$.  We know that in particular, $H_n^{dR}(S^n)\simeq \R$, and so a closed $m$-current on $S^n$ is unique up to scalar.  Hence $K=a \hat{S}^m$ for some $a\in \R$.
\end{proof}


\begin{lem}
	\label{Existence of closed polyhedral chain lemma}
	
	Let $J\in \hB_1(\C)$ be closed and compactly supported, and let $z\in \supp(J)^c$.  Then there exists $\epsilon > 0$ and a bounded open $W\subset \C$ such that
	\begin{enumerate}
		\item $W\cap B_\e(z)=\emptyset$, where $B_\epsilon (z)$ is the open $\epsilon$-ball about $z$,
		\item $J$ is supported in $W$,
		\item there exists a sequence of closed polyhedral chains $P_j \to J$ supported in $W$.
	\end{enumerate}
\end{lem}

\begin{proof} Choose $\e > 0$ such that  $B_{4\e}(z) \cap \supp(J) = \emptyset$.  Since $J$ is compactly supported,  $(\supp (J) \cup B_{4\e}(z))\subset B_R(0) \subset \C$ for some finite $R$.   Let $U = B_R(0) -  \overline{B_{3\e}(z)}$ and  $W= B_R(0)- \overline{B_\e}(z)$.  By construction, $J$ is supported in $U$ (and $W$).  Indeed, $W$ satisfies (1) and (2).

	\begin{figure}[htbp]
		\centering
			\includegraphics[height=2in]{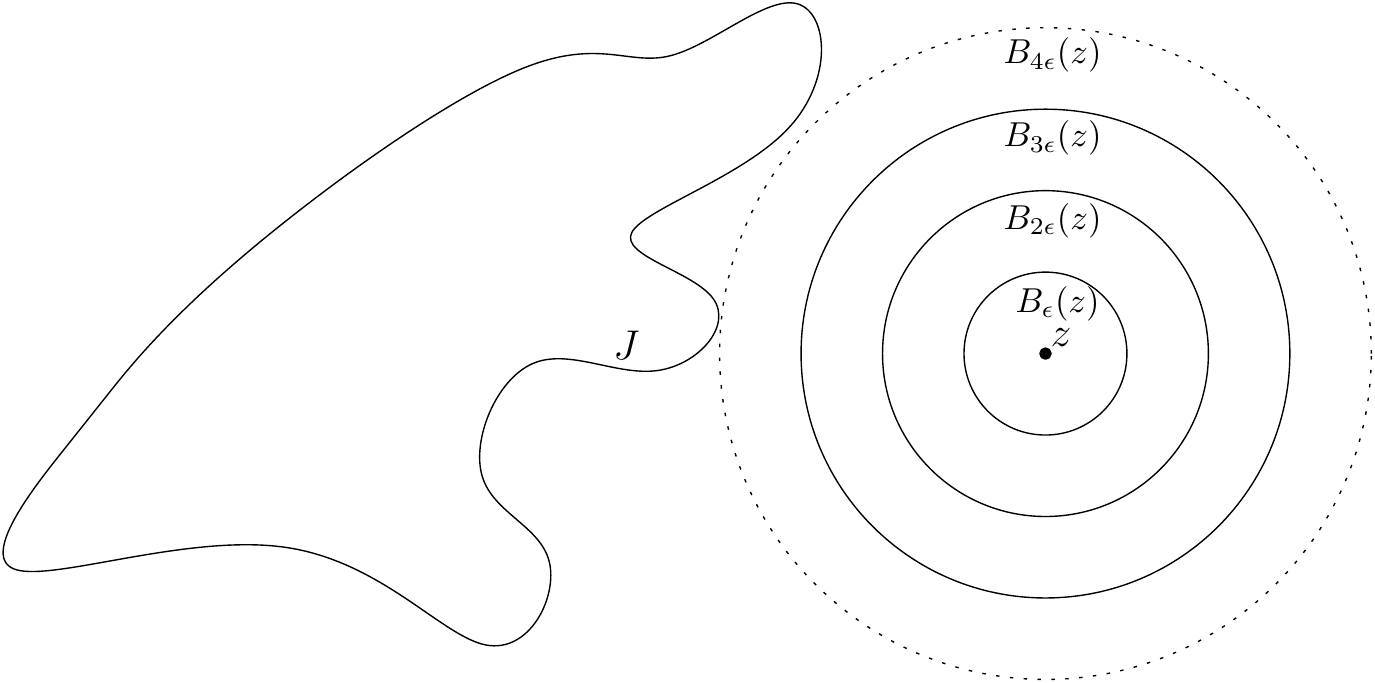}
		\caption{The setup}
		\label{fig:bullseye}
	\end{figure}

Let $K_j\rightarrow J$ be a sequence of polyhedral chains supported in $U$. We wish to replace $K_j$ with a sequence of closed polyhedral chains $P_j \to J$ supported in $W$.  Let $\pi$ be the projection of $\C\setminus \{z\}$ onto the  circle of radius $2\e$ about $z$,
\[
w\mapsto   2\e\frac{w-z}{\|w-z\|}+z.
\]

\begin{figure}[htbp]
	\centering
		\includegraphics[height=2.5in]{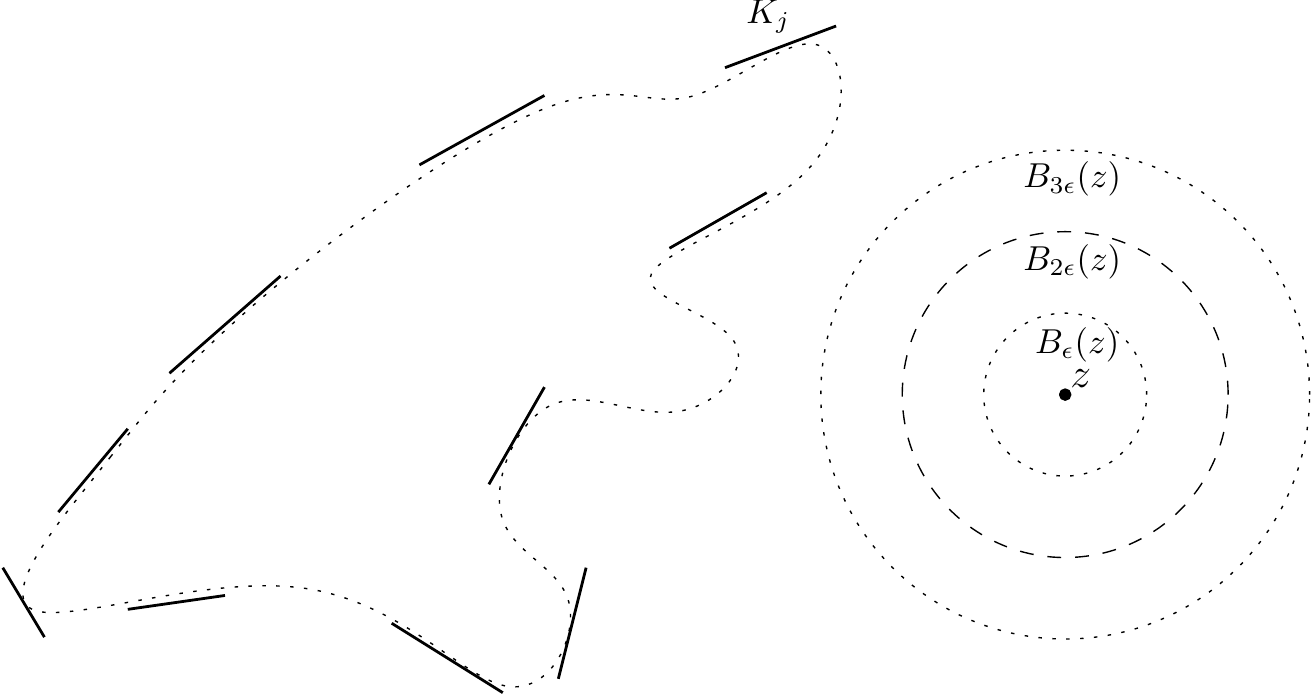}
	\caption{Starting with a polyhedral chain $K_j$}
	\label{fig:bullseye2}
\end{figure}
\newpage
Write $K_j=\sum_{j_m} k_{j_m}$, where the $k_{j_m}$'s are individual weighted 1-cells.  By splitting larger cells into smaller ones if necessary, we may assume without loss of generality that the lengths of the $k_{j_m}$'s are bounded by $1/j$.  Write $\p k_{j_m}=b_{j_m}-a_{j_m}$.  Since boundary commutes with pushforward,
\[
\p \pi_* k_{j_m}=\pi_* b_{j_m}-\pi_*a_{j_m}.
\]  

\begin{figure}[htbp]
	\centering
		\includegraphics[height=2.5in]{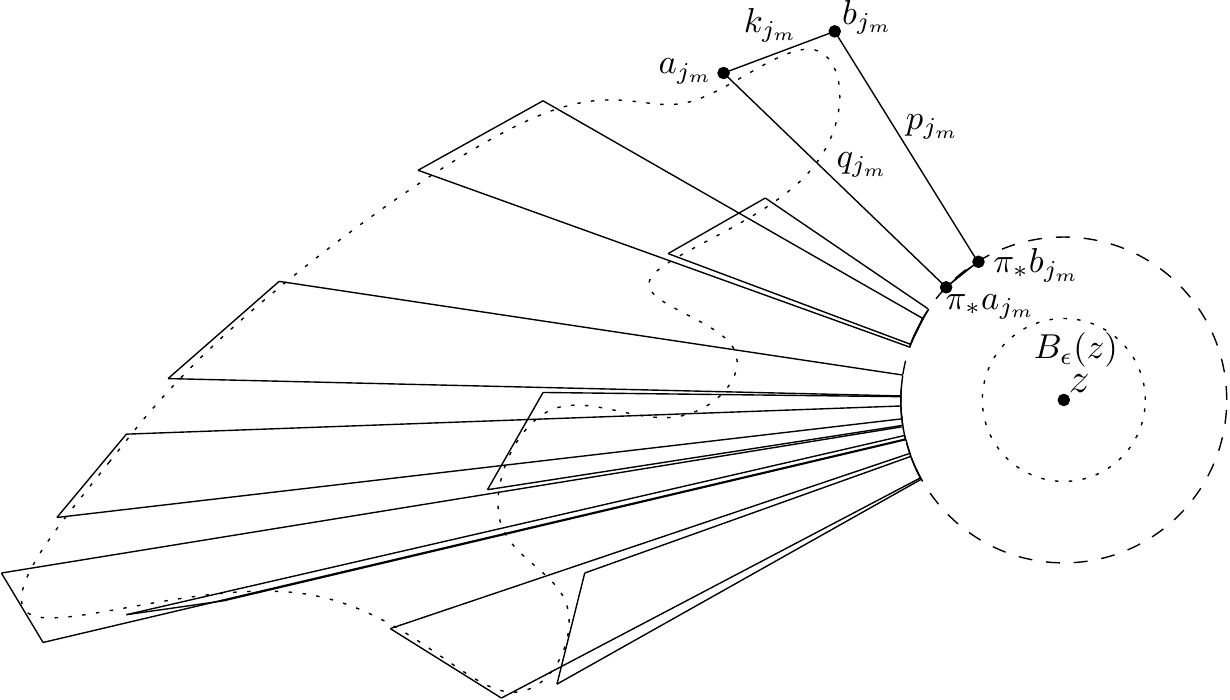}
	\caption{Modifying the $K_j$'s so that they become closed}
	\label{fig:bullseye3}
\end{figure}

We see that $a_{j_m}-\pi_* a_{j_m}$ bounds a cell $q_{j_m}$ and that $\pi_* b_{j_m}-b_{j_m}$ bounds a cell $p_{j_m}$.  Thus,
\[
r_{j_m}:=k_{j_m}-\pi_* k_{j_m}+q_{j_m}+p_{j_m}
\]

is a closed differential chain supported in $U$.  It is polyhedral, except for the term $\pi_* k_{j_m}$ which is a curved cell supported in the $2\e$-circle.   Let $R_j  = \sum_{}^{} r_{j_m}$. Then $\p R_j = 0$.  Since $\p$ is continuous, and $K_j \to J$,  we know $\p K_j \to \p J = 0$.  Since pushforward $\pi_*$ is linear and continuous on $\hB_1(U)$ (Corollary 10.1.3 of \cite{OCI}), $R_j  \to J - \pi_* J + \lim_{j \to \i} \sum q_{j_m}+p_{j_m}$. We next show that  $\lim_{j \to \i} \sum q_{j_m}+p_{j_m}  = 0$.  Since $K_j$ is compactly supported, we can write our truncated cone as a difference of two cones, namely,  $\textrm{cone}_z( \p K_j)- \textrm{cone}_z(\p \pi_* K_j)$.  Theorem 5.0.15 of \cite{poincarelemma} implies $\| \sum q_{j_m}+p_{j_m} \|_{B^r} \le 2R\|\p K_j\|_{B^r}  \to 0  $. 

Now we have a sequence of closed polyhedral chains $R_j$ supported in $W$ converging to the difference $J-\pi_*J$.  But since the image of $\pi$ is a circle, we have that $\pi_* J\in \hB_1(S^1)$, where $S^1$ here is embedded in $\C$ as the $2\e$-circle about $z$.  By Lemma \ref{topdim}, we know that $\pi_* J=a S$ where $a\in \R$.  Therefore,  $P_j = R_j + aS \to J$ and $P_j$ is closed.  Now $P_j$ is not quite polyhedral because some of its components are curved and supported on the $2\e$-circle, but we can easily replace such cells by straight ones that miss the $\e$-circle about $z$ due to the $1/j$ bound on the cells' length, and likewise can approximate $aS$ with a sequence of regular polygons. By our construction, the $P_j$ miss $B_\epsilon(z)$ for all $j>N$ for some $N$.
\end{proof}

\begin{thm}
	\label{Winding number constant on connected components}
	
If  $J\in\hB_1(\C)$ is closed and compactly supported, then $\mathrm{Ind}_J(z)$ is constant on connected components of $\supp(J)^c$.
\end{thm}

\begin{proof}
	
	Let $z\in \supp(J)^c$, and $\e$, $W$ and $P_n \to J$ as in Lemma \ref{Existence of closed polyhedral chain lemma}.  A closed polyhedral $1$-chain $P_n$ is just a sum of weighted piecewise linear parameterized closed curves.  That is, there exist piecewise linear parameterized closed curves $C_{n,i}$ and weights $\l_{n,i}\in \R$ such that
	\[
	\cint_{P_n}\o=\sum_i \l_{n,i} \int_{C_{n_i}}\o
	\]
	for all $\o\in \B_1(W)$.  Let $z_0\in B_\epsilon(z)$.  Then,
\[
\textrm{Ind}_J(z_0) =\frac{1}{2\pi i}\cint_J \frac{dw}{w-z_0} =  \lim_{n \to \i} \frac{1}{2\pi i}\cint_{P_n} \frac{dw}{w-z_0} = \lim_{n \to \i} \sum\l_{n,i}\textrm{Ind}_{C_{n,i}}(z_0).
\]
Since $\supp(P_n)\cap B_\epsilon(z)=\emptyset$, we know that $B_\epsilon(z)$ lies entirely within a connected component of $\supp(P_n)^c$.  Since the $C_{n,i}$ are piecewise smooth closed parameterized curves, the properties of the classical winding number hold.  In particular, $\textrm{Ind}_{C_{n,i}}(z_0)=\textrm{Ind}_{C_{n,i}}(z)$.  So,
\[
\lim_{n \to \i}\sum\l_{n,i}\textrm{Ind}_{C_{n,i}}(z_0)=\lim_{n \to \i}\sum\l_{n,i}\textrm{Ind}_{C_{n,i}}(z)=\frac{1}{2\pi i}\cint_{\lim_{n\rightarrow\i} P_n}\frac{dw}{w-z}=\frac{1}{2\pi i}\cint_{J} \frac{dw}{w-z}=\textrm{Ind}_J(z).
\]	
\end{proof}

\begin{cor}
	\label{Winding number is zero on unbounded connected component}
	
	If $J\in \hB_1(\C)$ is closed and compactly supported, and $z$ is in the unbounded connected component of $\supp(J)^c$, then $\textrm{Ind}_J(z)=0$.
\end{cor}

\begin{proof}
	
	By the selection of $W$ in Lemma \ref{Existence of closed polyhedral chain lemma}, $\cup_n \supp(P_n)$ is bounded, and so we may choose $z$ in the unbounded component of $\supp(J)^c$ so that $z$ is also in the unbounded component of $\supp(P_n)$ for all $n$.  As in the proof of Theorem \ref{Winding number constant on connected components}, the classical properties of winding number hold for $P_n$.  In particular,
	\[
	\textrm{Ind}_J(z)=\lim_{n\to\i}\textrm{Ind}_{P_n}(z)=0.
	\]
\end{proof}

We now know that our winding number behaves as it should.  However, we can say even more:

The \emph{part of a chain in an open set} $J\lfloor_U$ is defined in \S 3, Lemma 3.0.5 of \cite{poincarelemma} where $J \in \hB_k(\R^n)$ and $U \subset \R^n$ is an open set that can be written as a union of non-overlapping n-rectangles taken from the set of all rectangles whose faces do not lie on hyperplanes of a certain null set depending on $J$.  In particular,  $J_{\lfloor B_\e(x)}$ is well-defined a.e. $\e$. 

The \emph{mass} $M(J)$ of a differential $k$-chain $J \in \hB_k^r(\R^n)$ is given by  $M(J) :=    \inf \{ \liminf \|A_i\|_{B^s}: A_i \to J \mbox{ in } \hB_k^r\}$  in \cite{poincarelemma} \S 2. 
\begin{defn}
	\label{density}
	Let $K\in \hB_n(\R^n)$ such that $M(K)<\i$.  The \emph{signed density} of $K$ at the point $x\in\supp(\p K)^c$ is defined to be the value
	\[
	\lim_{\e\rightarrow 0}\frac{1}{\textrm{vol}(B_\e)}\cint_{K_{\lfloor B_\e(x)}}dv.
	\]
	
	This was shown to be well-defined in \cite{dclecture}.
\end{defn}

\begin{thm}
	\label{Winding number is equal to density}
	Let $J\in\hB_1(\C)$ be closed and compactly supported.  If $K\in \hB_2(\C)$ such that $\partial K=J$, $M(K)<\i$, and $z\in \supp(J)^c$, then $\textrm{Ind}_J(z)$ is equal to the signed density of $K$ at the point $z$.	
\end{thm}

\begin{proof}
	 Let $W$ be as in Lemma \ref{Existence of closed polyhedral chain lemma} and set $K$ equal to the $2$-chain constructed via the Poincar\'e lemma by coning $J$ over the point $z$.  From \cite{poincarelemma}, this choice of $K$ is unique. Set $K_n=\sum_{n_j}k_{n_j}\rightarrow K$, where $k_{n_j}$ are weighted 2-simplices with common base-point $z$.  We can choose this sequence $K_n$ such that the components of the boundaries of the $k_{n_j}$'s opposite $z$ are supported in $W$.  Let $m_{n_j}$ be the signed density of $k_{n_j}$, let $\theta_{n_j}$ be the angle subtended by $k_{n_j}$ at the point $z$ and let $l_{n_j}$ be the partial boundary of ${k_{n_j}}_{\lfloor B_\epsilon}$ opposite $z$.  Since $\lfloor_U$ and $\p$ are continuous, 
		\[
		\lim_{n\rightarrow \i}\sum_{n_j}l_{n_j}=\partial \left(K_{\lfloor B_\epsilon}\right).
		\]
		Since $\frac{1}{w-z}$ is holomorphic on a neighborhood of $\supp(K-K_{\lfloor B_\epsilon})$ (see \cite{poincarelemma}), it follows from Theorem
		\ref{Generalized Cauchy Integral Theorem} that
	\begin{align*}
		\textrm{Ind}_J(z)&=\frac{1}{2\pi i}\cint_{\partial (K_{\lfloor B_\epsilon})}\frac{dw}{w-z}\\
		&=\lim_{n\rightarrow\i}\sum_{n_j}\frac{1}{2\pi i}\cint_{l_{n_j}}\frac{dw}{w-z}\\
		&=\lim_{n\rightarrow\i}\sum_{n_j}\frac{m_{n_j}\theta_{n_j}}{2\pi},
	\end{align*}
	where the last integral is computed classically.  Likewise, the signed density of $K$ at $z$ is given by
	\begin{align*}
		\lim_{\epsilon\rightarrow 0}\frac{1}{\pi\epsilon^2}\cint_{K_{\lfloor B_\epsilon}} dx\,dy=\lim_{\epsilon\rightarrow 0}\lim_{n\rightarrow\i}\sum_{n_j}\frac{1}{\pi\epsilon^2}\cint_{{k_{n_j}}_{\lfloor B_\epsilon}}dx\,dy=\lim_{n\rightarrow\i}\sum_{n_j}\frac{m_{n_j}\theta_{n_j}}{2\pi},
	\end{align*}
	where the last integral is computed classically.  
\end{proof}

\begin{lem}              
	\label{Winding number zero implies poincare lemma for non-contractible sets}
	
	Suppose $J\in \hB_1(\C)$ is compactly supported and closed.  Suppose that $K\in \hB_2(\C)$ with $\p K= J$ satisfies $M(K)<\i$.  If $U$ is a bounded open set such that $J$ is supported in $U$ and if $\textrm{Ind}_J(w)=0$ for all $w\in U^c$, then $K$ is supported in $U$.
\end{lem}

\begin{proof}
	
	The signed density of $K$ is zero outside $U$, hence $K$ is supported in $U$ (proved in \cite{dclecture}), whereby the lemma follows from Theorem \ref{Winding number is equal to density}.
\end{proof}

\section{Global Cauchy Integral Theorem for Differential Chains}
\begin{thm}{Global Cauchy Integral Theorem for Differential Chains}
	\label{Generalized Global Cauchy Integral Theorem}
	
	Let $J\in \hB_1(\C)$ be closed and supported in a bounded open set $U\subset \C$ such that $\textrm{Ind}_J(w)=0$ for all $w\in U^c$.  Suppose there exists some $K\in \hB_2(\R^2)$ with $M(K)<\i$ and $\p K=J$. Then if $f$ is holomorphic on $U$,
	\[
	\cint_J f(z)dz=0.
	\]
\end{thm}

\begin{proof}
	
	By Lemma \ref{Winding number zero implies poincare lemma for non-contractible sets}, there exists some $K$ supported in $U$ such that $J=\partial K$.  The proof is otherwise identical to that of Theorem \ref{Generalized Cauchy Integral Theorem}.
\end{proof}

\section{Global Cauchy Integral Formula for Differential Chains}
\begin{thm}{Global Cauchy Integral Formula for Differential Chains}
	\label{Generalized Global Cauchy Integral Formula}
	
	Let $J\in \hB_1(\C)$ be closed and supported in a bounded open set $U\subset \C$ such that $\textrm{Ind}_J(w)=0$ for all $w\in U^c$.  Suppose there exists some $K\in \hB_2(\C)$ with $M(K)<\i$ and $\p K=J$.  Then if $f$ is holomorphic on $U$ and $z\in U\setminus \supp(J)$,
	\[
	f(z)\textrm{Ind}_J(z)=\frac{1}{2\pi i}\cint_J \frac{f(w)}{w-z}dw.
	\]
\end{thm}

\begin{proof}
	
	This follows from Theorem \ref{Generalized Global Cauchy Integral Theorem} in the same manner as Theorem \ref{Generalized Cauchy Integral Formula} followed from Theorem \ref{Generalized Cauchy Integral Theorem}.
\end{proof}

\section{Cauchy Residue Theorem for Differential Chains}
\begin{thm}{Cauchy Residue Theorem for Differential Chains}
	\label{Generalized Residue Theorem}
	
	Let $J\in \hB_1(\C)$ be closed and supported in a bounded open set $U\subset \C$ such that $\textrm{Ind}_J(w)=0$ for all $w\in U^c$.  Suppose there exists some $K\in \hB_2(\C)$ with $M(K)<\i$ and $\p K=J$.  Let $f$ be holomorphic in $U$ except for at finitely many points $a_k\in U\setminus\supp(J)$.  Then,
	\[
	\cint_J f(z)dz=\sum_k \textrm{Ind}_J(a_k)\cint_{B_k} f(z)dz,
	\]
	where $B_k=\p D_k$ and the $D_k\ni a_k$ are isolated open neighborhoods.
\end{thm}

\begin{proof}
	
	Since $\textrm{Ind}_z(J)=0$ for all $z\in U^c$, it follows from Lemma \ref{Winding number zero implies poincare lemma for non-contractible sets} that there exists $K$ supported in $U$ such that $\p K= J$.  For each $k$, let $D_k$ be an open ball around $a_k$ such that the closures of the $D_k$'s are disjoint from each other and contained in $U\setminus \supp(J)$.  Let $B_k=\p D_k$.  Let $D_k'\in \hB_2(\C)$ correspond canonically to $D_k$ and let $B_k'=\p D_k'$.
	
	By Theorems \ref{Winding number constant on connected components} and \ref{Winding number is equal to density}, the signed density of $K$ is constant on connected components of $U\setminus \supp(J)$.  It follows that the signed density of $K$ at the point $a_k$ is equal to the signed density of $K$ on any point in $D_k$.  Thus, $\supp(K-\sum_k \textrm{Ind}_J(a_k)D_k')=\supp(K)\setminus(\cup_k D_k)$.  Therefore, $f$ is holomorphic on a neighborhood of $\supp{K-\sum_k \textrm{Ind}_J(a_k)D_k'}$.  By Theorem \ref{Generalized Global Cauchy Integral Theorem},
	\[
	\cint_{\partial(K-\sum_k \textrm{Ind}_J(a_k)D_k')}f(z)dz=0.
	\]
	Therefore,
	\[
	\cint_J f(z)dz=\cint_{\partial K}f(z)dz=\cint_{\partial(\sum_k \textrm{Ind}_J(a_k)D_k')}f(z)dz=\sum_k \textrm{Ind}_J(a_k)\cint_{B_k'}f(z)dz.
	\]
\end{proof}          

\bibliography{bib}{}

\providecommand{\bysame}{\leavevmode\hbox to3em{\hrulefill}\thinspace}
\providecommand{\MR}{\relax\ifhmode\unskip\space\fi MR }
\providecommand{\MRhref}[2]{%
  \href{http://www.ams.org/mathscinet-getitem?mr=#1}{#2}
}
\providecommand{\href}[2]{#2}
\begin{thebibliography}{CMM82}

\bibitem[Har93]{harrison1}
Jenny Harrison, \emph{Stokes' theorem on nonsmooth chains}, Bulletin of the
  American Mathematical Society \textbf{29} (1993), 235--242.

\bibitem[Har98]{continuity}
\bysame, \emph{Continuity of the integral as a function of the domain}, Journal
  of Geometric Analysis \textbf{8} (1998), no.~5, 769--795.

\bibitem[Har07]{dclecture}
\bysame, \emph{{Lectures in Math 278 ``Topics in Analysis'' Berkeley}}, Spring
  2007.

\bibitem[Har10a]{poincarelemma}
\bysame, \emph{Geometric {Poincare Lemma}}, submitted, December 2010.

\bibitem[Har10b]{OCI}
\bysame, \emph{Operator calculus -- the exterior differential complex},
  submitted, December 2010.

\bibitem[HP10]{topological}
Jenny Harrison and Harrison Pugh, \emph{Topological aspects of differential
  chains}, Journal of Geometric Analysis \textbf{in press} (2010).

\bibitem[Poi92]{poincare}
Henri Poincar{\'{e}}, \emph{Sur les courbes d{\'{e}}finies par une
  {\'{e}}quation diff{\'{e}}rentielle}, Oeuvres \textbf{1} (1892).

\bibitem[Pug09]{thesis}
Harrison Pugh, \emph{Applications of differential chains to complex analysis
  and dynamics}, 2009, Harvard College senior thesis,
  \texttt{http://arxiv.org/abs/1012.5542}.


\end{thebibliography}
\bibliographystyle{amsalpha}
\end{document}